\documentclass{scrartcl}
\usepackage{graphicx}
\usepackage{listings}
\usepackage{microtype}
\usepackage{algorithm2e}
\usepackage{hyperref}
\usepackage{url}

\usepackage{amsmath}
\usepackage{amsthm}
\usepackage{amsfonts}
\usepackage{amssymb}

\usepackage{pgf}
\usepackage{pgfplots}
\pgfplotsset{compat=newest}

\DeclareMathOperator{\PV}{PV}

\newcommand{\Rplus}{\mathbb{R}_+}

\newtheorem{theorem}{Theorem}
\newcommand{\vett}[2]{\begin{bmatrix}#1\\#2\end{bmatrix}}

\newlength\plotheight
\newlength\plotwidth

\begin{document}

%\NLA{1}{6}{00}{28}{00}

%\runningheads{D.A.~Bini, B.~Meini, F.~Poloni}
%{On the solution of an equation arising in MBTs}

\title{On the solution of a quadratic vector equation arising in Markovian Binary Trees}
\author{Dario A. Bini\footnote{Dipartimento di Matematica, Universit\`a di Pisa. Largo Pontecorvo 5, 56127 Pisa, Italy. \url{{bini,meini}@dm.unipi.it}}, Beatrice Meini\textsuperscript{$\ast$} and Federico Poloni\footnote{Scuola Normale Superiore. Piazza dei Cavalieri 7, 56126 Pisa, Italy. \url{f.poloni@sns.it}}}

\maketitle

%\cgsn{Publishing Arts Research Council}{98--1846389}

%\noreceived{}
%\norevised{}
%\noaccepted{}

\begin{abstract}
  We present some advances, both from a theoretical and from a computational
  point of view, on a quadratic vector equation (QVE) arising in Markovian
  Binary Trees. Concerning the theoretical advances, some irreducibility
  assumptions are relaxed, and the minimality of the solution of the QVE is
  expressed in terms of properties of the Jacobian of a suitable function.
  From the computational point of view, we elaborate on the Perron
  vector-based iteration proposed in \cite{perronarxiv}.  In particular we
  provide a condition which ensures that the Perron iteration converges to the
  sought solution of the QVE.  Moreover we introduce a variant of the
  algorithm which consists in applying the Newton method instead of a
  fixed-point iteration. This method has the same convergence behaviour as the
  Perron iteration, since its convergence speed increases for
  close-to-critical problems.  Moreover, unlike the Perron iteration, the
  method has a quadratic convergence. Finally, we show that it is possible to
  alter the bilinear form defining the QVE in several ways without changing
  the solution.  This modification has an impact on convergence speed of the
  algorithms.

\end{abstract}

\newcommand{\keywords}[1]{\paragraph{Keywords:} #1}

\section{Introduction}
Markovian Binary Trees (MBTs) are a particular family of branching processes,
which are used to model the growth of populations consisting of several types
of individuals who evolve independently and may produce a variable number of offsprings during their lifetime.
MBTs have applications in biology, epidemiology and also in telecommunication
systems. We refer the reader to \cite{BeanKontoleonTaylor,HLR10} for definitions,
properties and applications.

One important issue related to MBTs is the computation of the extinction
probability of the population, which can be characterized as the minimal
nonnegative solution
\begin{align*}
x^\ast\in&\Rplus^N, &\text{with } \Rplus^N:=\{v \in \mathbb R^N : v_i \geq 0, i=1,\dots,N\},
\end{align*}
of the quadratic vector equation (QVE for short)
\begin{equation}\label{eq:qve}
x=a+b(x,x),
\end{equation}
where $a\in\Rplus^N$, $b: \Rplus^N\times \Rplus^N \to \Rplus^N$ is a
vector-valued bilinear form, such that the vector $e=(1,1,\dots,1)^T$ is always a solution of \eqref{eq:qve}. For writing the entries of $b$ in coordinates, we use the notation $b_{ijk}:=e_i^T b(e_j,e_k)$, where $e_\ell$ is the $\ell$th vector of the canonical basis. With this choice, 
\[
(b(x,y))_i=\sum_{j,k} b_{ijk} x_j y_k.
\]
In many papers the notation $b(s,t)=B(s \otimes t)$, with $B \in \Rplus^{N\times N^2}$ and $\otimes$ denoting the Kronecker product, is used instead; one can see that the two representations are equivalent. We favor the former, since it highlights the symmetry features of the problem. We mention the fact that the functions obtained by fixing the first or the second argument of the bilinear form, i.e., $b(y,\cdot)$ and $b(\cdot,z)$ for suitable $y,z\in \Rplus^N$, are linear maps from $\Rplus^N$ to itself, and thus they can be represented by $N\times N$ matrices with nonnegative entries.

The MBT is called subcritical, supercritical or critical if the spectral
radius $\rho(R)$ of the matrix
\begin{equation}\label{defR}
R: = b(e,\cdot)+ b(\cdot,e) 
\end{equation}
is strictly
less than one, strictly greater than one, or equal to one, respectively.

Under the stated assumptions, one can prove the existence of a minimal
nonnegative solution in the componentwise ordering. A proof using minimal
hypotheses is presented in \cite{qvearxiv}. In the subcritical and critical
cases the minimal nonnegative solution is the vector of all ones, while in the
supercritical case $x^\ast\le e$, $x^\ast\ne e$. Thus, only the supercritical
case is of interest for the computation of $x^\ast$.

Moreover, in the following we shall focus on the case in which $x^\ast>0$. It
is shown in \cite{qvearxiv} how to detect reliably the cases when this
property does not hold, and reduce them by projection to problems of lower
dimension with strictly positive minimal solution.

Several iterative methods have been proposed and analyzed for computing the
vector $x^\ast$. In \cite{BeanKontoleonTaylor} the authors propose two fixed
point iterations with linear convergence, called \emph{depth} and \emph{order}
algorithms.  Another linearly convergent algorithm, called \emph{thicknesses}
algorithm, is proposed in \cite{HLR10}. In \cite{HautphenneLatoucheRemiche}
and in \cite{HVH10} two variants of Newton's method are proposed.
A different algorithm, based on a Perron vector iteration, is proposed in
  \cite{perronarxiv}. This algorithm, unlike classical iterative methods,
  increases its convergence speed for close to critical problems.

  In this paper we provide theoretical and computational advances concerning
  the QVE \eqref{eq:qve}. We first show that the matrix $R$ of \eqref{defR}
  can be assumed irreducible, since if it were reducible, we may reduce the
  problem of solving \eqref{eq:qve} to the problem of solving QVEs of smaller
  dimension, whose associated matrix $R$ is irreducible.  Assuming that $R$ is
  irreducible, we provide a new characterization of the minimal nonnegative
  solution $x^\ast$, in terms of the properties of the Jacobian of the
  function $F(x)=x-a-b(x,x)$, evaluated at $x=x^\ast$.  This property, which
  complements the results in \cite{qvearxiv},  allows
  us to give a condition which ensures that the limit of the Perron
  vector-based iteration provides the sought solution $x^\ast$ of the
  quadratic vector equation. 

  Moreover, we introduce a variant of the Perron vector-based iteration, which
  consists in applying the Newton method instead of a fixed-point iteration.
  This method is quadratically convergent, and has the same convergence
  behaviour as the Perron iteration, as its convergence speed increases for
  close-to-critical problems. The number of iterations needed by this variant
  is usually lower than the number of iterations needed by the original Perron
  iteration. However, due to the larger complexity of the single iteration
  step, the Newton-based method is generally slower than the Perron iteration,
  in terms of total computational time.

  Finally, we show that it is possible to alter the bilinear $b(x,y)$ form
  definining the QVE in several ways without changing the solution.  This
  modification has an impact on convergence speed: in most examples, making
  the wrong choice can double the number of iterations needed. We show that,
  at least on the experiments reported, the best results are given by a
  symmetrization of the original bilinear form.

  The paper is organized as follows. In Section \ref{classical} we recall
  classical algorithms based on fixed point iterations, while in Section
  \ref{perron} we recall the Perron-based iteration.  In Section
  \ref{sec:rred} we discuss the case where the matrix $R$ of \eqref{defR} is
  reducible, and we reduce the QVE to smaller size QVEs whose associated
  matrix $R$ is irreducible. In Section \ref{sec:min} the minimality of the
  solution $x^\ast$ of the QVE is expressed in terms of properties of the
  Jacobian of the function $F(x)$ at $x=x^\ast$. This result is used in
  Section \ref{sec:lim} to ensure that the limit of the Perron-based iteration
  provides the sought solution $x^\ast$. The Newton version of the
  Perron-based iteration is proposed in Section \ref{sec:newt}.  In Section
  \ref{sec:choice} the choice of the bilinear form $b(x,y)$ is discussed. The
  results of the numerical experiments are presented and discussed in Section
  \ref{sec:ne}. We draw conclusions in Section \ref{sec:conc}.

\section{Classical iterations}\label{classical}
Several iterative methods have been proposed and analyzed for computing the vector $x^\ast$. In
\cite{BeanKontoleonTaylor} the authors propose two
iterations with linear convergence, called \emph{depth} and \emph{order} algorithms, which are defined respectively by the two linear equations
\begin{align*}
(I-b(\cdot,x_{k}))x_{k+1}&=a,\\
(I-b(x_{k},\cdot))x_{k+1}&=a.
\end{align*}
The \emph{thicknesses}
algorithm, still linearly convergent, is proposed in \cite{HLR10} and consists in alternating iterations of each of the two above methods.

In \cite{HautphenneLatoucheRemiche} the authors apply the Newton method to the map
\begin{equation}\label{F}
 F(x):=x-a-b(x,x),
\end{equation}
obtaining the iteration defined by
\begin{equation}\label{cn}
 (I-b(x_k,\cdot)-b(\cdot,x_k))x_{k+1}=a-b(x_k,x_k), 
\end{equation}
which converges quadratically. Its convergence speed is usually much higher than that of the previous, linearly-convergent iterations.
A modification of the Newton method, which increases slightly its convergence speed, has been proposed in \cite{HVH10}.

All these methods have probabilistic interpretations, in that their $k$-th iterate $x_k$ can be interpreted as the probability of extinction of the process restricted to a special subtree $\mathcal T_k$. Each of them
provides a sequence $\{x_k\}_k$ of nonnegative vectors, with
$x_0=(0,\ldots,0)^T$, which converges monotonically to the minimal nonnegative
solution $x^\ast$. A common feature of all these methods is that their
convergence speed slows down when the problem, while being supercritical,
gets {\em close to critical}, i.e., the vector $x^\ast$ approaches the vector
of all ones. This happens because the mean extinction time increases, and thus sampling larger and larger trees is needed to capture the behavior of the iteration.

\section{A Perron-vector-based iteration}\label{perron}
In \cite{perronarxiv}, the authors propose an iterative scheme based on a
different interpretation. Let us suppose for now that the nonnegative matrix
$R$, as defined in \eqref{defR}, is irreducible --- we discuss this assumption
in Section~\ref{sec:rred}. Then, since irreducibility depends only on the zero
pattern of the matrix, $b(u,\cdot)+b(\cdot,v)$ is irreducible for each $u,v
\in \Rplus^N$ with strictly positive entries.

If we set $y=e-x$, equation \eqref{eq:qve} becomes
\begin{equation}\label{eqy}
y=b(y,e)+b(e,y)-b(y,y).
\end{equation}
A special solution of
\eqref{eqy} is $y^\ast=e-x^\ast$, where $x^\ast$ is the minimal nonnegative solution of
\eqref{eq:qve}. Notice that $0\leq y^\ast\leq e$.
In the probability interpretation of Markovian Binary Trees, since $x^\ast$ represents the
extinction probability, then $y^\ast=e-x^\ast$ can be interpreted as survival probability. In
particular,  $y^\ast_i$ is the probability that a
colony starting from a single individual in state $i$ does not become extinct
in a finite time. The three summands in the right-hand side of \eqref{eqy} also admit an interesting probabilistic interpretation \cite{perronarxiv}.

If we set $H_y:=b(\cdot,e)+b(e-y,\cdot)$, equation \eqref{eqy} becomes
\begin{equation}\label{eqHy}
y=H_y y.
\end{equation}
If $H_y$ is nonnegative and irreducible (which happens for sure if $y<e$, in view of the irreducibility of $R$), then the Perron-Frobenius theorem implies that $\rho(H_{y^\ast})=1$ and $y^\ast$ is the Perron vector of the matrix $H_{y^\ast}$.

This interpretation allows to design new algorithms for computing $y^\ast$ and
$x^\ast$. Applying a functional iteration directly to \eqref{eqHy}, or the
Newton method, gives nothing new, since we just did a change of
variable. However, if we define the map $\PV(M)$ as the Perron vector of a
nonnegative irreducible matrix $M$, we may rewrite \eqref{eqHy} as
\begin{equation} \label{eqPV}
 y=\PV(H_y).
\end{equation}
We may apply a fixed-point iteration to solve \eqref{eqPV}, thus generating a
sequence $\{y_k\}_k$ of positive vectors such that the vector $y_{k+1}$ is the
Perron vector of the matrix $H_{y_k}$, i.e.,
\begin{equation} \label{eqPVk}
 y_{k+1}=\PV(H_{y_k}).
\end{equation}

A suitable normalization of the Perron vector, consistent with the solution, is needed to obtain a well-posed iteration.
An optimal normalization choice is suggested in \cite{perronarxiv}. If we take $w$ as the Perron vector of the nonnegative irreducible matrix $R^T$, then we may normalize $y_{k+1}$ so that
\begin{equation}\label{cannorm}
 w^T(y_{k+1}-b(y_{k+1},e)-b(e,y_{k+1})+b(y_{k+1},y_{k+1}))=0,
\end{equation}
i.e., we impose that the residual of \eqref{eqHy} for $y=y_{k+1}$ is
orthogonal to $w$. With this choice, one can prove \cite{perronarxiv} that the
convergence speed of the sequence $\{y_k\}_k$ defined in \eqref{eqPVk}, with
the normalization condition \eqref{cannorm}, is linear with a small
convergence factor for close-to-critical problems, and tends to superlinear as
the considered problems approach criticality.  Thus, although the convergence
of this method is linear, surprisingly its speed \emph{increases} as the
problem gets close to critical, unlike the classical iterations.

\section{Dealing with reducible $R$}\label{sec:rred}

The following result shows that when $R$ is reducible we can reduce a QVE to two smaller-dimension problems to be solved successively with a kind of back-substitution. Therefore, for the solution of a generic QVE, we only need to apply the Perron iteration to the case in which $R$ is irreducible.

\begin{theorem}
Suppose that, for a QVE \eqref{eq:qve} with $x^\ast>0$, we have
\[
 R=\begin{bmatrix}
    R_{11} & R_{12}\\
    0 & R_{22}
   \end{bmatrix},
\]
where $R_{11}$ is $M\times M$ and $R_{22}$ is $(N-M)\times(N-M)$.
 Let
\begin{align*}
 x=&\begin{bmatrix}x_1\\x_2\end{bmatrix},& x^\ast=&\begin{bmatrix}x_1^\ast\\x_2^\ast\end{bmatrix}, & a=&\begin{bmatrix}a_1\\a_2\end{bmatrix}
\end{align*}
be partitioned accordingly. Let
\begin{align*}
 P:=&\begin{bmatrix}
    I_M & 0_{M\times (N-M)}
   \end{bmatrix},&
 Q:=&\begin{bmatrix}
    0_{(N-M)\times M} & I_{N-M}
   \end{bmatrix},
\end{align*}
be the orthogonal projection on the first $M$ and last $N-M$ components respectively.
Let us define the bilinear form on $\Rplus^{N-M}$
\[
 b_2(u,v):=Qb(Q^Tu,Q^Tv).
\]
Moreover, for each $y\in \Rplus^{M+N}$, let us define
\begin{align*}
 T_{y}:=&I_M-Pb(\cdot,Q^Ty)P^T-Pb(Q^Ty,\cdot) P^T,&
 a_{y}:=&{T_{y}}^{-1}(a_1+Pb(Q^Ty,Q^Ty)),
\end{align*}
and the bilinear form on $\Rplus^M$
\[
 b_{y}(u,v):={T_{y}}^{-1}Pb(P^Tu,P^Tv).
\]
Then,
\begin{enumerate}
 \item $x$ solves \eqref{eq:qve} if and only if $T_{x_2}$ is nonsingular, and its block components $x_2$ and $x_1$ solve respectively the two quadratic vector equations
\begin{equation}\label{due}
 x_2=a_2+b_2(x_2,x_2)
\end{equation}
and
\begin{equation}\label{uno}
 x_1=a_{x_2}+b_{x_2}(x_1,x_1).
\end{equation}
\item If $x^\ast$ is the minimal solution to \eqref{eq:qve}, then $x_1^\ast$ and $x_2^\ast$ are the minimal solution to \eqref{uno} and \eqref{due} respectively.
\end{enumerate}

\end{theorem}
\begin{proof}
First notice that since $P$ and $Q$ are projections on complementary subspaces, \eqref{eq:qve} holds if and only if it holds when projected on both, i.e.,
\begin{subequations}
\begin{align}
  x_1=&a_1+Pb(x,x), \label{Pproj}\\
  x_2=&a_2+Qb(x,x). \label{Qproj}
\end{align}
\end{subequations}
Since $R_{ij}=\sum_{k=1}^N (b_{ijk}+b_{ikj})$, it follows from the block structure of $R$ (and from $b_{ijk}\geq 0$) that $b_{ijk}=0$ whenever $i>M$ and either $j\leq M$ or $k\leq M$. This implies that the second block row of $b(u,v)$ depends only on the second block rows of $u$ and $v$. We can write this formally as $Qb(u,v)=Qb(Q^TQu,Q^TQv)$.
Then, \eqref{Qproj} is equivalent to \eqref{due}.

By exploiting bilinearity and the fact that $P^TP+Q^TQ=I_N$, we can rewrite \eqref{Pproj} as
\begin{align*}
 x_1=a_1+Pb(P^Tx_1,P^Tx_1)+Pb(P^Tx_1,Q^Tx_2)+Pb(Q^Tx_2,P^Tx_1)+Pb(Q^Tx_2,Q^Tx_2),
\end{align*}
or
\begin{align*}
 T_{x_2}x_1=a_1+Pb(Q^Tx_2,Q^Tx_2)+Pb(P^Tx_1,P^Tx_1).
\end{align*}
Since the right-hand side is nonnegative and $x_1$ is positive, the Z-matrix $T_{x_2}$ is an M-matrix. Therefore, we may invert it to get \eqref{uno}.
The steps in the above proof can be reversed provided $T_{x_2}$ is nonsingular, thus the reverse implication holds as well.

Let us now prove the second part of the theorem. Equation \eqref{due} admits a minimal solution due to the general existence theorem (since it admits at least a solution); suppose it is $x_2\neq x_2^\ast$; then, by minimality, $x_2\leq x_2^\ast$. The matrix $T_{x_2} \geq T_{x_2^\ast}$ is an M-matrix, thus $T_{x_2^\ast}^{-1} \geq T_{x_2}^{-1}$. Therefore, $a_{x_2}\leq a_{x_2^\ast}$ and $b_{x_2}\leq b_{x_2^\ast}$.
We have
\[
 x_1^\ast=a_{x_2^\ast}+b_{x_2^\ast}(x_1^\ast,x_1^\ast) \geq a_{x_2}+b_{x_2}(x_1^\ast,x_1^\ast),
\]
thus the equation \eqref{uno} has a supersolution, and this implies that it has a solution by \cite[Lemma~5]{qvearxiv}. Let $x_1$ be its minimal solution; then, $x$ is a solution to \eqref{eq:qve} by the first part of this theorem, but this is in contradiction with the minimality of $x^\ast$, since $x_2\lneqq x_2^\ast$.
Therefore $x_2^\ast$ is the minimal solution to \eqref{due}. If \eqref{uno} admitted a solution $x_1\lneqq x_1^\ast$, then by the first part of the theorem
\[
 \vett{x_1}{x_2^\ast}
\]
would be a solution to \eqref{eq:qve}, and this again contradicts the minimality of $x^\ast$.
\end{proof}

\medskip
Let $F'_x:=I-b(x,\cdot)-b(\cdot,x)$ be the Jacobian of the map $F(x)$ defined
in \eqref{F} (see \cite{HautphenneLatoucheRemiche,qvearxiv}).  Notice that if
$x>0$, then $F'_x$ has the same positivity pattern as $R$, and thus is
irreducible whenever $R$ is. Moreover, when $x=e$ is a solution to
\eqref{eq:qve}, then the all-ones vectors of suitable dimension solve
\eqref{due} and \eqref{uno}, thus the Perron vector-based iteration can be
applied to the reduced problems as well.

\section{An alternative characterization of minimality}\label{sec:min}
The following theorem provides a practical criterion to check the minimality
of a solution. 
\begin{theorem}\label{minimalmmatrix}
Let $x>0$ be a solution of \eqref{eq:qve} and assume that $R$ is irreducible. Then, $F'_x$ is an M-matrix if and only if $x$ is minimal.
\end{theorem}
\begin{proof}
 The implication ($x^\ast$ minimal) $\Rightarrow$ ($F'_{x^\ast}$ is an M-matrix) has been proved in \cite{qvearxiv}. We prove the reverse here. The proof is split in two different arguments, according to whether $F'_x$ is a singular or nonsingular M-matrix. 

 Let $F'_x$ be a nonsingular M-matrix, and let $\bar x$ be another nonnegative
 solution;
 we need to prove that $\bar x-x\geq 0$. From the Taylor expansion of $F(x)$
 (and the fact that $F''\leq 0$) we have
\[
 0=F(\bar x)=F(x)+F'_x(\bar x-x)+\frac12F''_x(\bar x-x,\bar x-x) \leq F'_x(\bar x-x),
\]
that is, $F'_x(\bar x-x)\geq 0$. It suffices to multiply by $(F'_x)^{-1}\geq 0$ to get $\bar x-x\geq 0$, as needed. 

Let now $F'_x$ be a singular M-matrix. Suppose that $x$ is not minimal, and $x^\ast \lneqq x$ is the minimal solution to \eqref{eq:qve}. Then, $F'_{x^\ast}\gneqq F'_{x}$ is a (singular or nonsingular) M-matrix, by the reverse implication of this theorem. Thus, by the properties of M-matrices, $F'_x$ must be a nonsingular M-matrix, which is a contradiction.
\end{proof} 

Notice that this characterization of minimality allows to deduce easily the fact, claimed above, that the solution $e$ is minimal only in the subcritical and critical cases.

\section{On the limit of the Perron iteration}\label{sec:lim}
The following result shows that, under reasonable assumptions, the limit of the Perron vector-based iteration is the minimal solution of \eqref{eq:qve}.
\begin{theorem}\label{perronlimit}
  Suppose that $R$ is irreducible, and that $x^\ast>0$. Suppose that the
  Perron iteration \eqref{eqPVk}, with normalizing condition \eqref{cannorm},
  converges to a vector $y^\ast$ such that $y^\ast \leq e$. Then, $x=e-y^\ast$
  is the minimal solution of \eqref{eq:qve}.
\end{theorem}
\begin{proof}
 Let us first prove that the spectral radius of $H_{y^\ast}$ is 1.  The iterates of the Perron iteration satisfy
\begin{subequations}\label{perronit}
 \begin{gather}
 \lambda_{k+1} y_{k+1}=H_{y_k}y_{k+1},  \\
 w^T(y_{k+1}-H_{y_{k+1}}y_{k+1})=0.
 \end{gather}
\end{subequations}

By passing \eqref{perronit} to the limit, we get
\begin{subequations}
 \begin{gather}
 \lambda^\ast y^\ast=H_{y^\ast} y^\ast, \label{perronlim1} \\
 w^T(y^\ast-H_{y^\ast}y^\ast)=0. \label{perronlim2}
 \end{gather}
\end{subequations}
Notice that $\lambda^\ast$ is well-defined, as it may be defined as the common ratio between the components of $H_{y^\ast} y^\ast$ and those of $y^\ast$. We left-multiply  \eqref{perronlim1} by $w^T$ to get $w^T(\lambda^\ast y^\ast-H_{y^\ast} y^\ast)=0$, which, compared to \eqref{perronlim2}, tells us that $\lambda^\ast=1$. In particular, this implies that $x=e-y^\ast$ is a solution of \eqref{eq:qve}, as we may verify directly by back-substitution.

Moreover, $\rho(H_{y^\ast})=1$, and thus
$I-H_{y^\ast}=I-b(e-y^\ast,\cdot)-b(\cdot,e)$ is a singular M-matrix. Thus the
Z-matrix $F'_x=I-b(e-y^\ast,\cdot)-b(\cdot,e-y^\ast)\geq
I-b(e-y^\ast,\cdot)-b(\cdot,e)$ is an M-matrix, too. By
Theorem~\ref{minimalmmatrix}, this implies that $x=e-y^\ast$ is minimal.

\end{proof}

\section{The Perron--Newton method}\label{sec:newt}
We may also apply Newton's method for the solution of \eqref{eqPV}. 

We first recall the following result from \cite{perronarxiv}, which provides
an explicit form for the Jacobian of the iteration map $G(y)$ defining the
iteration \eqref{eqPVk} with the normalization \eqref{cannorm}, i.e.,
\begin{align*}
 G(y):=\text{the Perron vector of $H_y$, normalized s.t. } w^T\left(G(y)-H_{G(y)}G(y)\right)=0.
\end{align*}

\begin{theorem}
Let $y$ be such that $H_y$ is nonnegative and irreducible.
Let $u=G(y)$, and let
$v$ be such that $v^TH_y=\lambda v^T$, where $\lambda=\rho(H_y)$.
Then the Jacobian of the map $G$ at $y$ is
\begin{equation}\label{jac}
 JG_y=\left(I-\frac{u\sigma_1^T}{\sigma_1^T u}  \right) (H_y-\lambda I)^\dagger  \left(I-\frac{uv^T}{v^Tu}\right)b(\cdot,u),
\end{equation}
where 
\[
\sigma_1^T:=w^T(I-b(e-u,\cdot)-b(\cdot,e-u))
\]
and the symbol $\vphantom{M}^\dagger$ denotes the Moore--Penrose pseudo-inverse.
\end{theorem}
With the aid of this formula, we may define the Perron--Newton method for the solution of \eqref{eq:qve} as in Algorithm~\ref{algo:pn}.
\begin{algorithm}
\caption{The Perron--Newton algorithm}\label{algo:pn}
\SetKwInOut{Input}{input}
\Input{the bilinear form $b$ (note that $a$ is not necessary --- in fact it can be deduced from $e=a+b(e,e)$)}
\Input{the normalization vector $w>0$ (a good choice is taking the Perron vector of $R^T$, see \cite{perronarxiv})}
$y\leftarrow e$\;
\While{a suitable stopping criterion is not satisfied}{
$u\leftarrow G(y)$\;
$J\leftarrow JG_y$ (computed using \eqref{jac})\;
$y\leftarrow y-(I-J)^{-1}(y-u)$\;
}
\uIf{$0\leq y \leq e$}
{$x\leftarrow e-y$\;}
\Else{(error: no convergence)\;}
\end{algorithm}

A step of Newton's method basically requires a step of the Perron vector-based fixed-point iteration associated with \eqref{eqPV}, followed by the computation of a Moore--Penrose pseudoinverse and the solution of a linear system. Thus its cost is larger than, but still comparable to, the cost of a step of the Perron vector-based functional iteration. This is compensated by the fact that the Newton method has quadratic convergence, and thus requires less iterations. 

The convergence properties of the Perron--Newton method for close-to-critical
problems are similar to those of the Perron vector-based functional iteration.
For close-to-critical
problems one has $x^\ast \approx e$, therefore $\rho(JG_{y^\ast}) \approx
0$. Hence, by the Newton--Kantorovich theorem
\cite{OrtegaRheinboldt} there is convergence for sufficiently
close-to-critical problems. The proof of Theorem~\ref{perronlimit} can be
easily adapted to show that the limit point must correspond to the minimal
solution of \eqref{eq:qve} if $0\leq y^\ast \leq e$. Moreover, since
$\rho(JG_{y^\ast}) \approx 0$, the matrix to invert is well-conditioned and
$y-G(y)$ has a simple zero.

\section{On the choice of the bilinear form $b$}\label{sec:choice}
Equation~\eqref{eq:qve}, and thus its solution, depend only on the
quadratic form $b(t,t):=B(t\otimes t)$; however, there are different ways to extend it to a (nonnecessarily symmetric)
bilinear form $b(s,t)$. Namely, for each $i$ and each $j\neq k$, we may alter simultaneously $b_{ijk}$ and $b_{ikj}$, as long as their sum remains invariant, and they both remain positive. For example, we may switch the two terms in every such pair, obtaining the bilinear form $b^T(s,t):=b(t,s)$.

Some of the solution algorithms depend essentially on the
choice of the bilinear extension: for instance, the \emph{depth} and
\emph{order} algorithms. It is easy to see that the \emph{depth} algorithm
applied to $b^T$ coincides with \emph{order} applied to $b$, and vice
versa. Instead, in the classical Newton's method \eqref{cn}, the bilinear form
appears only in the expressions $b(x_k,\cdot)+b(\cdot,x_k)$ and $b(x_k,x_k)$,
which are unaffected by this change. Thus the classical Newton method stays the same no matter which bilinear extension we choose.

On the other hand, one can see that the Perron-vector based functional
iteration and its Newton based version do depend on the bilinear extension,
and their convergence speed is affected by this choice. The expression of the
bilinear form ultimately reflects a modeling aspect of the problem. While in
the original definition of a branching process an individual splits into two
new ones in two different states, it is often convenient to identify one as
the ``mother'' and one as the ``child'', even if this distinction is
artificial.  In fact, we can safely redefine who is the mother and who is the
child, as long as we do not change the total probability that an individual in
state $i$ originates two offsprings in states $j$ and $k$. This corresponds
exactly to changing the bilinear form $b$ in the described way.

Among the possibilities for the modifications of $b$, we list the following.
\begin{description}
 \item[Transposition] $b^T(s,t):=b(t,s)$
 \item[Symmetrization] $b^S(s,t):=\frac12 \left(b(s,t)+b^T(s,t)\right)$
 \item[Desymmetrization 1] 
\[ (b^{D1})_{ijk}:=\begin{cases}
             b_{ijk}+b_{ikj} & \text{if $j<k$}\\
	     b_{ijk} & \text{if $j=k$}\\
	     0 & \text{if $j>k$}\\
            \end{cases}
\]
 \item[Desymmetrization 2] 
\[ (b^{D2})_{ijk}:=  \left(b^{T}\right)^D1=\begin{cases}
             b_{ijk}+b_{ikj} & \text{if $j>k$}\\
	     b_{ijk} & \text{if $j=k$}\\
	     0 & \text{if $j<k$}\\
            \end{cases}
\]
\end{description}
In the following section, we report numerical experiments performed with the above bilinear extensions and compare the computational times. We do not have a definitive result on which choice gives the best convergence: as is the case with the \emph{depth} and \emph{order} algorithms, the optimal bilinear extension may vary in different instances of the problem.

\section{Numerical experiments}\label{sec:ne}
We performed numerical experiments to assess the speed of the proposed methods. The tests were performed on a laptop (Intel Pentium M 735 1.70Ghz) with Matlab R2010a and considered two sample parameter-dependent problems.
\begin{description}
 \item[P1] a small-size Markovian binary tree with branches of varying length, described in \cite[Example~1]{HautphenneLatoucheRemiche}. It is an MBT of size $N=9$
depending on a parameter $\lambda$, which is critical for $\lambda \approx
0.85$ and supercritical for larger values of $\lambda$.
 \item[P2] a random-generated MBT of larger size ($N=100$). It is created by generating a random bilinear form $b$, choosing a suitable $a$ so that $a+b(e,e)=Ke$ for some $K$, and then scaling both $a$ and $b$ in order to eliminate $K$. We report the Matlab code used for its generation in Algorithm~\ref{algo:randmbt}.
\begin{algorithm}
 \caption{Generating a random MBT}\label{algo:randmbt}
\SetKwInOut{Input}{input}
\Input{the size $N$ of the MBT and a parameter $\lambda>0$}
\lstinline{e=ones(N,1)}\;
\lstinline{rand('state',0)}\;
\lstinline{b=rand(N,N*N)}\;
\lstinline{K=max(b*kron(e,e))+$\lambda$}\;
\lstinline{a=K*e-b*kron(e,e)}\;
\lstinline{a=a/K}\;
\lstinline{b=b/K}\;
\end{algorithm}
Larger choices of the parameter $\lambda$ increase the values of $a$, i.e., the probability of immediate death, and thus enlarge the extinction probability making the process closer to critical. With $N=100$, the process is critical for $\lambda \approx 4920$.
\end{description}

\begin{figure}
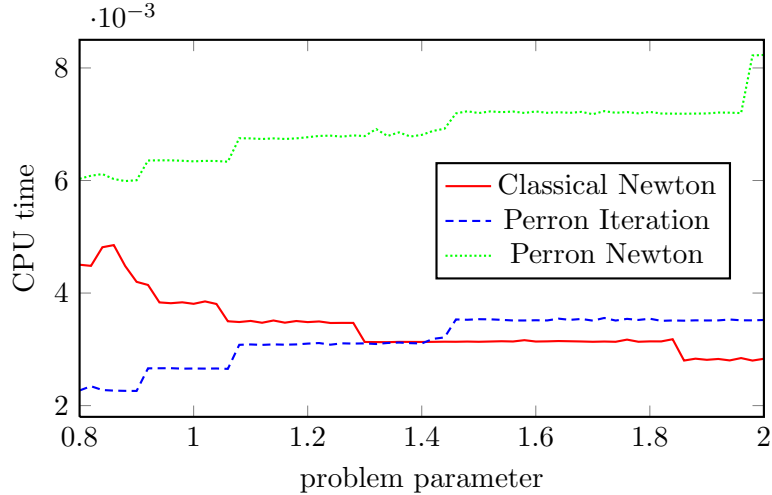

 \begin{center}
   \setlength\plotheight{5cm}
  \setlength\plotwidth{9cm}
%  \beginpgfgraphicnamed{E1-slides-out}

\tikzstyle{every pin}=[fill=white,
draw=black,
font=\footnotesize]
\pgfplotsset{every axis/.append style={thick}}


%  \endpgfgraphicnamed
 \end{center}
 \caption{CPU time vs. parameter $\lambda$ for P1 --- lower=better} \label{p1}
\end{figure}

\begin{figure}
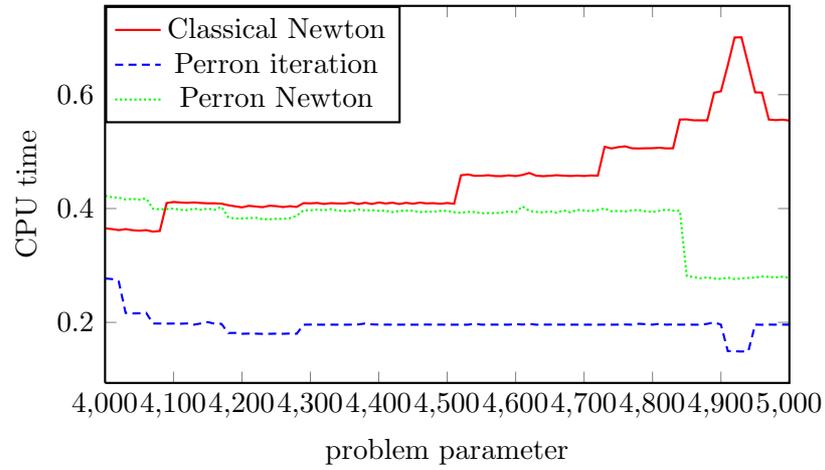

 \begin{center}
   \setlength\plotheight{5cm}
  \setlength\plotwidth{9cm}
%  \beginpgfgraphicnamed{E1-slides-out}

\tikzstyle{every pin}=[fill=white,
draw=black,
font=\footnotesize]
\pgfplotsset{every axis/.append style={thick}}

%
%  \endpgfgraphicnamed
 \end{center}
 \caption{CPU time vs. parameter $\lambda$ for P2} \label{p2}
\end{figure}

\begin{figure}
 \centering
 \setlength\plotheight{5cm}
 \setlength\plotwidth{9cm}
%
\caption{Number of steps needed for the Perron iteration for P1 with several variants of the bilinear form}\label{b1}
\end{figure}

\begin{figure}
 \centering
 \setlength\plotheight{5cm}
 \setlength\plotwidth{9cm}
%
\caption{Number of steps needed for the Perron iteration for P2/ with several variants of the bilinear form}\label{b2}
\end{figure}

Figure~\ref{p1} shows a plot of the computational times for classical Newton and the two Perron vector-based methods for different values of the parameter $\lambda$. Depth, order and thicknesses are not reported in the graph as they are much slower than these methods, as also shown by the experiments in \cite{HautphenneLatoucheRemiche}. While in close-to-critical
cases the time for CN has a spike, the ones for PN and PI seems to decrease. While having in theory worse convergence properties, the Perron iteration is faster than the Perron Newton method: the additional overhead of the pseudoinverse and of the computation of both left and right dominant eigenvector do not make up for the increased convergence rate.

Figure~\ref{p2} shows the corresponding plot for the larger problem P2. We point out that two different methods were used to compute the Perron vectors in the two problems. For P2, we use \lstinline{eigs}, which is based on an Arnoldi method \cite{arpack}. On the other hand, for P1, due to the really small size of the problem, it is faster to compute a full eigenvector basis with \lstinline{eig} and then select the Perron vector.

With this choice, both the Perron iteration and Perron--Newton method are faster than the classical Newton method on this larger-size problem, in the close-to-critical region.

Figure~\ref{b1} reports the number of iteration (which essentially grows as the CPU time) for the Perron iteration on P1 with several alternative bilinear forms equivalent to $b$. We see that among the two possible ``branch switches'', in this example the iteration with $b$ converges faster than the one with $b^T$. Clearly this cannot be a general result: due to the involutory nature of this transposition operation, if we started with $\tilde b:=b^T$, then the faster choice would have been $\tilde b^T=(b^T)^T=b$. Thus we cannot infer a rule for telling which of the two is preferable. Similarly, it is impossible to do a proper comparison among $b^{D1}$ and $B^{D2}$. On the other hand, an interesting result is that the performance of the iteration with $b^S$ seems to be on par with the better of the two.

Figure~\ref{b2} reports the same comparison for the problem P2. The results are less pronounced than on the previous example: since the entries of the bilinear form $b$ are generated randomly, the difference between the ``left'' and ``right'' branches of the binary tree should be less marked than in P1, where the two directions are willingly unbalanced. Nevertheless, the symmetrized bilinear form consistently yields slightly lower iteration counts. 

Therefore, based on these results, we suggest to apply the Perron iteration and Newton methods on the symmetrized bilinear form instead of the original one.

\section{Conclusions}\label{sec:conc}
In this paper we presented several possible implementation variants of the Perron vector-based iteration introduced in \cite{perronarxiv}. A Newton method based on the same formulation of the problem is slightly less effective than the original iteration, although it maintains the same good convergence properties for close-to-critical problem. Moreover, we highlight the fact that there is a family of possible modifications to the bilinear form $b$ that alter the form of solution algorithms, but not the original equation \eqref{eq:qve} and its solution. One of these modifications, the symmetrization, seems to achieve better results than the original formulation of the numerical algorithms. 

Moreover, we present a couple of theoretical results on quadratic vector
equations that show how to ensure that the obtained solution is the desired
one, and how to deal with the problems in which an irreducibility assumption
is not satisfied.

\bibliographystyle{wileyj} \bibliography{paper_nsmc}

\end{document}